\documentclass[12pt, reqno]{amsart}
\usepackage{amsmath, amsthm, amscd, amsfonts, amssymb, graphicx, color}
\usepackage[bookmarksnumbered, colorlinks, plainpages]{hyperref}

\newtheorem{theorem}{Theorem}[section]
\newtheorem{lemma}[theorem]{Lemma}
\newtheorem{proposition}[theorem]{Proposition}
\newtheorem{corollary}[theorem]{Corollary}
\theoremstyle{definition}

\theoremstyle{remark}
\newtheorem{remark}[theorem]{Remark}
\numberwithin{equation}{section}
\def\Tr{\mathop{\rm Tr}}

\begin{document}
\setcounter{page}{1}

\title[Inequalities for arithmetic mean and harmonic mean]{Matrix inequalities for the difference between arithmetic mean and harmonic mean}

\author[W.S. Liao, J.L. Wu]{Wenshi Liao$^1$$^{*}$ and Junliang Wu$^2$}

\address{$^{1}$ College of Mathematics and Statistics, Chongqing University, Chongqing, 401331, P.R. China.}
\email{\textcolor[rgb]{0.00,0.00,0.84}{liaowenshi@gmail.com}}

\address{$^{2}$ College of Mathematics and Statistics, Chongqing University, Chongqing, 401331, P.R. China.}
\email{\textcolor[rgb]{0.00,0.00,0.84}{jlwu678@163.com}}


\subjclass[2010]{Primary 15A45; Secondary 15A60, 26E60.}

\keywords{Arithmetic mean, harmonic mean, matrix inequalities, positive definite matrices.}

\date{Received: xxxxxx; Revised: yyyyyy; Accepted: zzzzzz.
\newline \indent $^{*}$ Corresponding author}

\begin{abstract}
Motivated by the refinements and reverses of arithmetic-geometric mean and arithmetic-harmonic mean inequalities for scalars and matrices, in this article, we generalize the scalar and matrix inequalities for the difference between arithmetic mean and harmonic mean. In addition, relevant inequalities for the Hilbert-Schmidt norm and determinant are established.
\end{abstract} \maketitle

\section{Introduction}
Let $M_n(\mathbb{C})$ be the space of $n\times n$ complex matrices. $I$ stands for the identity matrix. The Hilbert-Schmidt norm ($l_2$ norm, Frobenius norm or Schur norm) of $A=[a_{ij}]\in M_n(\mathbb{C})$ is defined by
\[\left\| A\right\|_F^2=\left(\sum\limits_{i=1}^n\sum\limits_{j=1}^n |a_{ij}|^2\right)^{1/2}=\left(\Tr|A|^2\right)^{1/2}=\left(\sum\limits_{i=1}^ns_i^2(A)\right)^{1/2},\]
where $\Tr$ is the trace functional, $|A|=(A^*A)^{1/2}$ and $s_1(A)\ge s_2(A)\ge\cdots\ge s_n(A)$ denote the singular values
 of $A$ (the eigenvalues of positive semi-definite matrix $|A|$) arranged in decreasing order and repeated according
  to multiplicity (see \cite[p.341-342]{Horn}). It is well-known that each unitarily invariant norm is a symmetric guage function of singular values
  \cite[p.91]{Bhatia}, so the Hilbert-Schmidt norm is unitarily invariant.

For $a,b>0$, $v\in[0,1]$ and $t\in \mathbb{R}$, the power mean
\[
M_t(v;a,b)=\left(va^t+(1-v)b^t\right)^{1/t}, t\neq 0,
\]
makes a path of means from the harmonic mean at $t=-1$ to the arithmetic mean $t=1$ via the geometric mean at $t\rightarrow 0$ and $M_0(v;a,b)=\lim\limits_{t\rightarrow 0}M_t(v;a,b)$. If $s\le t$, then $M_s(v;a,b)\le M_t(v;a,b)$ and the two means are equal if and only if $a=b$ (see \cite[p.194-196]{Roberts}). So
\[
M_{-1}(v;a,b)\le M_0(v;a,b) \le M_1(v;a,b),
\]
that is,
\begin{equation}
\label{in11}
\left(va^{-1}+(1-v)b^{-1}\right)^{-1}\le a^vb^{1-v}\le va+(1-v)b.
\end{equation}
Note that this is the classical arithmetic-geometric-harmonic mean inequalities and it's worthwhile to mention that the second one is the Young inequality, for more details about the refinements and reverses of the Young inequality, the reader is referred to \cite{Alzer, Furuichi, Kittaneh, Zuo}.

The following is a noncommutative matrices version of the arithmetic-geometric-harmonic mean inequalities which is the main result from \cite{Ando1, Ando2} (Also see \cite{Sagae}):
For positive definite matrices $A,B\in M_n(\mathbb{C})$ and $0\le v\le 1$,
\begin{equation}
\label{in1}
\left(vA^{-1}+(1-v)B^{-1}\right)^{-1}\le A^{\frac{1}{2}}(A^{-\frac{1}{2}}BA^{-\frac{1}{2}})^v A^{\frac{1}{2}}\le vA+(1-v)B.
\end{equation}
For convenience, we use the following
notations to define the weighted arithmetic mean, geometric mean and harmonic mean for scalars and matrices:
\[
A_v(a,b)=va+(1-v)b,~
H_v(a,b)=\left(va^{-1}+(1-v) b^{-1}\right)^{-1},
\]
\[
A\nabla _v B=vA+(1-v) B, A\# _v B=A^{\frac{1}{2}}(A^{-\frac{1}{2}}BA^{-\frac{1}{2}})^v A^{\frac{1}{2}},
\]
\[
 A!_v B=\left( {vA^{-1}+(1-v) B^{-1}} \right)^{-1},
\]
where $a,b>0$, $0\le v\le 1$ and $A,B\in M_n(\mathbb{C})$ are positive definite matrices. When $v =\frac{1}{2}$, we write $A(a,b)$, $H(a,b)$, $A\nabla B$, $A\#B$ and $A!B$ for brevity, respectively. The above notations and definitions will be valid throughout the whole paper.

It is evident that the full matrix algebra of all $n\times n$ matrices with entries in the complex field is the finite-dimensional
 case of the $C^*$-algebra of all bounded linear operators on a complex separable Hilbert space.
If one inequality is valid for positive invertible operators, so is valid for positive definite matrices.

Motivated by Furuichi's refinement of the Young inequality for positive invertible operators $A,B$ and $v \in [0,1]$ (see \cite{Furuichi})
\begin{equation}
\label{in2}
A\nabla _v B\ge A\#_v B+2\min \left\{ {v ,1-v } \right\}\left(A\nabla B-A\#B\right),
\end{equation}
and Kittaneh and Manasrh's reverse Young inequality for two positive definite matrices $A,B$ and $v \in [0,1]$ (see \cite{Kittaneh})
\begin{equation}
\label{in3}
A\nabla _v B\le A\#_v B+2\max \left\{ {v ,1-v } \right\}\left(A\nabla B-A\#B\right),
\end{equation}
Hirzallah et al. \cite{Hirzallah1} generalized the inequality \eqref{in2} and \eqref{in3}: For positive invertible operators $A,B$ and ${\rm\textbf{p}}=(p_1, p_2)\in \mathbb{R}_+^2$,
if $A\ge B,0<p_1\le p_2$ or $A\le B,0<p_2\le p_1$, then
\begin{equation}
\label{in4}
A\nabla _{\frac{p_1}{p_1+p_2}} B\ge A\#_{\frac{p_1}{p_1+p_2}} B+\frac{4p_1p_2}{(p_1+p_2)^2}\left(A\nabla B-A\#B\right),
\end{equation}
in addition, if $A\ge B,0<p_2\le p_1$ or $A\le B,0<p_1\le p_2$, then the inequality \eqref{in4} is reversed.

Zuo et al.\cite{Zuo} refined the weighted
arithmetic-harmonic mean inequality and extended it to two invertible positive operator $A,B$ as follows:
\begin{equation}
\label{zuo1}
A_v(a,b)\ge H_v(a,b)+ 2\min \left\{ {v ,1-v } \right\}\left(A(a,b)-H(a,b)\right),
\end{equation}
\begin{equation}
\label{zuo2}
A\nabla _v B\ge A!_v B+2\min \left\{ {v ,1-v } \right\}\left(A\nabla B-A!B\right).
\end{equation}
Krni\'{c} et al.\cite{Krni} presented a reverse of the inequality \eqref{zuo2}
\begin{equation}
\label{kr1}
A\nabla _v B\le A!_v B+2\max \{1-v ,v \}(A\nabla B-A!B).
\end{equation}
Apparently, the following scalars inequality is also valid:
\begin{equation}
\label{kr2}
A_v(a,b)\le H_v(a,b)+2\max \{1-v ,v \}\left(A(a,b)-H(a,b)\right).
\end{equation}

Our main task is to improve the scalar and matrix inequalities of the difference between arithmetic mean and harmonic mean.
This article is organized in the following way: in Section 2, we derive several new weighted arithmetic-harmonic mean inequality.
In Section 3, we extend inequalities proved in Section 2 from the scalars setting to a matrix-algebra setting.
 In Section 4 and 5, the Hilbert-Schmidt norm and determinant inequalities for positive definite matrices are established.

\section{Inequalities of the difference between arithmetic mean and harmonic mean}
The first theorem is our main result about the difference between arithmetic and harmonic mean which extends the inequalities \eqref{zuo1} and \eqref{kr2}.
\begin{theorem}
Let $v$, $\tau$ and $\lambda$ be real numbers with $0<v<\tau<1$ and $\lambda\ge1$. Then
\begin{equation}
\label{th21}
\left(\frac{v}{\tau}\right)^{\lambda}<M_{v,\tau;\lambda}(a,b)<\left(\frac{1-v}{1-\tau}\right)^{\lambda}
\end{equation}
holds for all positive and distinct real numbers $a$ and $b$, where $M_{v,\tau;\lambda}(a,b)=\frac{A_v(a,b)^{\lambda}-H_v(a,b)^{\lambda}}{A_\tau(a,b)^{\lambda}-H_\tau(a,b)^{\lambda}}$.
Moreover, $\lim\limits_{a\rightarrow 0}M_{v,\tau;\lambda}(a,b)=\left(\frac{1-v}{1-\tau}\right)^{\lambda}$ and $\lim\limits_{a\rightarrow \infty}M_{v,\tau;\lambda}(a,b)=\left(\frac{v}{\tau}\right)^{\lambda}$.
\end{theorem}

\begin{proof}
Let $\lambda>0$, $0<v<\tau<1$ and $0<x\neq 1$. We define
\begin{align*}
F(v;\lambda;x)&=\frac{A_v(x,1)^{\lambda}-H_v(x,1)^{\lambda}}{v^{\lambda}}
=\left(x+\frac{1}{v}-1\right)^{\lambda}-\left(\frac{v^2}{x}+v-v^2\right)^{-\lambda},\\
G(v;\lambda;x)&=\frac{A_v(x,1)^{\lambda}-H_v(x,1)^{\lambda}}{(1-v)^{\lambda}}=\left(\frac{v x}{1-v}+1\right)^{\lambda}-\left(\frac{(1-v)v}{x}+(1-v)^2\right)^{-\lambda}.
\end{align*}

The function $F(v)=F(v;\lambda;x)$ on $(0,1)$ is differentiable. Its partial differentiation at $v$ can be expressed as
\begin{align*}
\frac{\partial}{\partial v}F(v;\lambda;x)&=
\frac{\lambda}{v^{\lambda+1}H_v(x,1)^{1-\lambda}}\left[H_v(x,1)^2\left(\frac{2v}{x}+1-2v\right)-\left(\frac{A_v(x,1)}{H_v(x,1)}\right)^{\lambda-1}\right]\\
&\le \frac{\lambda}{v^{\lambda+1}H_v(x,1)^{1-\lambda}}\left[H_v(x,1)^2\left(\frac{2v}{x}+1-2v\right)-1\right]\\
&=\frac{\lambda}{v^{\lambda+1}H_v(x,1)^{1-\lambda}}\cdot\frac{-v^2\left(x^{-1}-1\right)^2}{\left(vx^{-1}+(1-v)\right)^2}\\
&<0.
\end{align*}
This implies that $F(v;\lambda;x)$ is strictly decreasing with respect to $v$. Hence, for $0<v<\tau<1$, we obtain
\begin{equation}
\label{th212}
\frac{A_\tau(x,1)^{\lambda}-H_\tau(x,1)^{\lambda}}{\tau^{\lambda}}<\frac{A_v(x,1)^{\lambda}-H_v(x,1)^{\lambda}}{v^{\lambda}}.
\end{equation}

The partial differentiation of $G(v;\lambda;x)$ at $v$ yields

\begin{align*}
&\frac{\partial}{\partial v}G(v;\lambda;x)\\
&=\frac{\lambda x}{(1-v)^{\lambda+1}H_v(x,1)^{1-\lambda}}\left[\left(\frac{A_v(x,1)}{H_v(x,1)}\right)^{\lambda-1}+H_v(x,1)^2\left(\frac{1-2v}{x^2}+\frac{2v-2}{x}\right)\right]\\
&\ge\frac{\lambda x}{(1-v)^{\lambda+1}H_v(x,1)^{1-\lambda}}\left[1+H_v(x,1)^2\left(\frac{1-2v}{x^2}+\frac{2v-2}{x}\right)\right]\\
&=\frac{\lambda x}{(1-v)^{\lambda+1}H_v(x,1)^{1-\lambda}}\cdot\frac{(1-v)^2\left(x^{-1}-1\right)^2}{\left(vx^{-1}+(1-v)\right)^2}\\
&>0.
\end{align*}
Thus, $G(v;\lambda;x)$ is strictly increasing with respect to $v$. For $0<v<\tau<1$, we obtain
\begin{equation}
\label{th213}
\frac{A_\tau(x,1)^{\lambda}-H_\tau(x,1)^{\lambda}}{(1-\tau)^{\lambda}}<\frac{A_v(x,1)^{\lambda}-H_v(x,1)^{\lambda}}{(1-v)^{\lambda}}.
\end{equation}
Next, let $x=a/b$ in \eqref{th212} and \eqref{th213} and multiplying both sides by $b^{\lambda}$, then we obtain \eqref{th21}.

Further, we have
\begin{align*}
M_{v,\tau;\lambda}(a,b)&=\frac{A_v(a,b)^{\lambda}-H_v(a,b)^{\lambda}}{A_\tau(a,b)^{\lambda}-H_\tau(a,b)^{\lambda}}\\
&=\frac{(va+(1-v)b)^{\lambda}-\left(va^{-1}+(1-v)b^{-1}\right)^{-\lambda}}{(\tau a+(1-\tau )b)^{\lambda}-\left(\tau a^{-1}+(1-\tau )b^{-1}\right)^{-\lambda}}\\
&=\frac{(va+(1-v)b)^{\lambda}-\left[ab(vb+(1-v)a)^{-1}\right]^{\lambda}}{(\tau a+(1-\tau )b)^{\lambda}-\left[ab(\tau b+(1-\tau)a)^{-1}\right]^{\lambda}} ,
\end{align*}
which satisfies
\begin{equation}
\label{th214}
\lim\limits_{a\rightarrow 0}M_{v,\tau;\lambda}(a,b)=\frac{((1-v)b)^{\lambda}-0}{((1-\tau)b)^{\lambda}-0}=\left(\frac{1-v}{1-\tau}\right)^{\lambda},
\end{equation}
and the representation
\[
M_{v,\tau;\lambda}(a,b)=\frac{A_v(1,\frac{b}{a})^{\lambda}-H_v(1,\frac{b}{a})^{\lambda}}{A_\tau(1,\frac{b}{a})^{\lambda}-H_\tau(1,\frac{b}{a})^{\lambda}}
\]
leads to
\begin{equation}
\label{th215}
\lim\limits_{a\rightarrow \infty}M_{v,\tau;\lambda}(a,b)=\frac{v^{\lambda}-0}{\tau^{\lambda}-0}=\left(\frac{v}{\tau}\right)^{\lambda}.
\end{equation}
The limit relations \eqref{th214} and \eqref{th215} reveal that the upper and lower bounds given in \eqref{th21} are sharp.
\end{proof}

\begin{remark}
If $\lambda=1$ and $\tau=1/2$ in the inequalities \eqref{th21}, then
\[
2v(A(a,b)- H(a,b))\le A_v(a,b)- H_v(a,b)\le 2(1-v)(A(a,b)- H(a,b)),
\]
which is equivalent to \eqref{zuo1} and \eqref{kr2}, respectively.

If $\lambda=2$ and $\tau=1/2$ in the inequalities \eqref{th21}, then
\begin{equation}
\label{rem21}
\begin{split}
4v^2(A(a,b)^2- H(a,b)^2)&\le A_v(a,b)^2- H_v(a,b)^2\\
&\le 4(1-v)^2(A(a,b)^2- H(a,b)^2).
\end{split}
\end{equation}
\end{remark}

Next, we establish another type of upper and lower bounds for the difference of $A_v(a,b)$ and $H_v(a,b)$. We need the following lemma (see \cite{Alzer}).
\begin{lemma}\label{lemma21}
Let $v\in(0,1)$ and $f:[a,b]\rightarrow \mathbb{R}$ be twice differentiable with $-\infty <m\le f''(x)\le M<+\infty$ for all $x\in (a,b)$. Then
\begin{align*}
\frac{v(1-v)}{2}(b-a)^2m&\le vf(a)+(1-v)f(b)-f(va+(1-v)b)\\
&\le\frac{v(1-v)}{2}(b-a)^2M.
\end{align*}
The factor $v(1-v)$ is the best possible.
\end{lemma}

\begin{theorem}\label{theorem22}
Let $v\in(0,1)$ and $a,b>0$ with $a<b$. Then we have
\begin{equation}
\label{th22}
v(1-v)\left(1-\frac{a}{b}\right)^2a\le A_v(a,b)-H_v(a,b)\le v(1-v)\left(1-\frac{b}{a}\right)^2b.
\end{equation}
\end{theorem}

\begin{proof}
If we take $f(x)=1/x$ in Lemma \ref{lemma21}, then we have
\begin{align*}
v(1-v)\left(b-a\right)^2\frac{1}{b^3}&\le va^{-1}+(1-v)b^{-1}-(va+(1-v)b)^{-1}\\
&\le v(1-v)\left(b-a\right)^2\frac{1}{a^3}.
\end{align*}
Replace $a$ by $b^{-1}$ and $b$ by $a^{-1}$ in the above inequalities, respectively,
\begin{align*}
v(1-v)\left(b^{-1}-a^{-1}\right)^2a^3&\le vb+(1-v)a-(vb^{-1}+(1-v)a^{-1})^{-1}\\
&\le v(1-v)\left(b^{-1}-a^{-1}\right)^2b^3,
\end{align*}
then we can obtain the desired result by replacing $v$ by $1-v$.

Setting $t=b/a>1$, then \eqref{th22} leads to
\[
v(1-v)\le g_v(t)\le v(1-v)\frac{1}{t},
\]
where
\[
g_v(t)=\frac{v+(1-v)t-(v+(1-v)t^{-1})^{-1}}{\left(1-t\right)^2}.
\]
Since $\lim\limits_{t\rightarrow 1}g_v(t)=v(1-v),$ the limit reveals that the factor $v(1-v)$ is sharp.
\end{proof}

\section{Matrix versions of the difference between arithmetic mean and harmonic mean }
In this section, we begin with the following lemma which is based on the inequalities \eqref{th21}
and the spectral theorem for Hermitian matrices.
\begin{lemma}\label{lemma31}
Let $Q\in M_n(\mathbb{C})$ be positive definite and $v$, $\tau$ be real numbers with $0<v\le \tau<1$. Then
\begin{equation}
\label{lm31}
\frac{v}{\tau}(I\nabla_\tau Q-I!_\tau Q)\le I\nabla_v Q-I!_v Q \le \frac{1-v}{1-\tau}(I\nabla_\tau Q-I!_\tau Q).
\end{equation}
\end{lemma}

\begin{proof}
By the spectral theorem, see \cite[Theorem 2.5.6]{Horn}, there exist a unitary
 matrix $U$ such that $Q=UDU^*$, where $D={\rm diag}(\mu_1,\mu_2,\cdots,\mu_n)$ with the eigenvalues $\mu_i> 0, i=1,2,\cdots,n$ of $Q$.
 Applying \eqref{th21} with $\lambda=1$, we have
\[
\frac{v}{\tau}(1\nabla_\tau \mu_i-1!_\tau \mu_i)\le 1\nabla_v \mu_i-1!_v \mu_i \le \frac{1-v}{1-\tau}(1\nabla_\tau \mu_i-1!_\tau \mu_i).
\]
For diagonal matrix $D$, the above inequality can be written as
\[
\frac{v}{\tau}(I\nabla_\tau D-I!_\tau D)\le I\nabla_v D-I!_v D \le \frac{1-v}{1-\tau}(I\nabla_\tau D-I!_\tau D).
\]
Using the fact that any $\ast$-conjugation preserves the L\"{o}ewner partial order between Hermitian matrices, see \cite[Theorem 7.7.2]{Horn}.
We obtain \eqref{lm31} by applying the $\ast$-conjugation $\bullet\mapsto U\bullet U^*$ to the identity matrix $I$ and the diagonal matrix $D$.
\end{proof}

The next theorem generalize the inequalities \eqref{zuo2} and \eqref{kr1}.

\begin{theorem}
Let $A, B\in M_n(\mathbb{C})$ be positive definite. If $v, \tau$ are real numbers with $0<v\le \tau<1$, then
\begin{equation}
\label{th31}
\frac{v}{\tau}(A\nabla_\tau B-A!_\tau B)\le A\nabla_v B-A!_v B \le \frac{1-v}{1-\tau}(A\nabla_\tau B-A!_\tau B).
\end{equation}
\end{theorem}

\begin{proof}
The matrices $A^{-1/2}$ and $A^{1/2}$ are positive definite under the condition that $A$ is positive definite.
The result follows by putting $Q=A^{-1/2}BA^{-1/2}$ in Lemma \ref{lemma31} and applying the $\ast$-conjugation $\bullet\mapsto A^{1/2}\bullet A^{1/2}$ to it, see \cite[Theorem 7.7.2]{Horn}.
\end{proof}

\begin{corollary}
Let $A, B\in M_n(\mathbb{C})$ be positive definite. If $v$ is a real number with $0<v\le1/2$, then
\begin{equation}
\label{cor31}
2v(A\nabla B-A! B)\le A\nabla_v B-A!_v B \le 2(1-v)(A\nabla B-A!B).
\end{equation}
\end{corollary}
\begin{proof}
Let $\tau=1/2$ in the inequality \eqref{th31}.
\end{proof}

Note that the inequalities \eqref{cor31} are equivalent to \eqref{zuo2} and \eqref{kr1} proved by Zuo et al. \cite{Zuo} and Krni\'{c} et al. \cite{Krni}.
\quad

Based on Theorem \ref{theorem22}, we establish another type of upper bound for the difference of arithmetic mean and harmonic mean for positive definite matrices
which is only related to one argument.
\begin{theorem}
Let $A, B\in M_n(\mathbb{C})$ be positive definite satisfy $0<mI\le A\le B \le MI$. If $v$ is real number with $0\le v \le 1$, then
\begin{equation}
\label{th32}
 A\nabla_v B-A!_v B\le v(1-v)\left(1-\frac{M}{m}\right)^2B.
\end{equation}
\end{theorem}

\begin{proof}
From the inequalities \eqref{th22} with $t=b/a>1$, we have
\[
 v+(1-v)t-(v+(1-v)t^{-1})^{-1}\le v(1-v)\left(1-t\right)^2 t,
\]
for $0\le v\le 1$. Thus we have the following inequality for the positive definite matrix $I\le T\le M/mI$,
\[
 v+(1-v)T-(v+(1-v)T^{-1})^{-1}\le v(1-v)\max\limits_{1\le t\le\frac{M}{m} }\left(1-t\right)^2 T.
\]
Since $I\le A^{-\frac{1}{2}}BA^{-\frac{1}{2}}\le M/m I$, putting $T=A^{-\frac{1}{2}}BA^{-\frac{1}{2}}$ in the above inequality, we deduce
\begin{align*}
 v+(1-v)A^{-\frac{1}{2}}BA^{-\frac{1}{2}}&-\left(v+(1-v)\left(A^{-\frac{1}{2}}BA^{-\frac{1}{2}}\right)^{-1}\right)^{-1}\\
&\le v(1-v)\left(1-\frac{M}{m}\right)^2 A^{-\frac{1}{2}}BA^{-\frac{1}{2}},
\end{align*}
then multiplying the both sides by $A^{\frac{1}{2}}$, we obtain the inequalities \eqref{th32}.
\end{proof}

\section{The Hilbert-Schmidt norm inequalities for arithmetic mean and harmonic mean }

Based on the refinements and reverses of the Young inequality, Hirzallah and Kittaneh \cite{Hirzallah} and Kittaneh and Manasrah \cite{Kittaneh} had showed that if $A, B, X\in M_n(\mathbb{C})$ with $A$ and $B$  positive definite, for $v\in [0,1]$, then
\begin{align*}
\min \{v^2,(1-v)^2 \}\left\| AX-XB\right\|_F^2&\le \left\| {vAX +(1 - v)XB} \right\|_F^2-\left\|A^vX B^{1 - v} \right\|_F^2,\\
\max \{v^2,(1-v)^2\}\left\| AX-XB\right\|_F^2&\ge \left\| {vAX +(1 - v)XB} \right\|_F^2-\left\|A^vX B^{1 - v} \right\|_F^2.
\end{align*}

From the above inequalities, it is evident that the following inequality holds:
\[
\left\| {vAX +(1 - v)XB} \right\|_F^2\ge\left\|A^vX B^{1 - v} \right\|_F^2,
\]
and the authors of \cite{Salemi} pointed out that if $A$ and $B$ are positive semidefinite and $X\in M_n(\mathbb{C})$, then for unitarily invariant norm $\||A^vX B^{1 - v} \||\le \||vAX +(1 - v)XB\||$ does not holds in general.

Inspired by the above Hilbert-Schmidt norm versions of the improved Young inequalities, we derive the following theorem about the difference-type inequalities between arithmetic and harmonic mean for the Hilbert-Schmidt norm by the inequality \eqref{th21} with $\lambda=2$.
\begin{theorem}\label{theorem41}
Let $A, B, X\in M_n(\mathbb{C})$ such that $A$ and $B$ be positive definite. If $v, \tau$ are real numbers with $0<v\le \tau<1$, then
\begin{equation}
\label{th41}
\left(\frac{v}{\tau}\right)^2 \le\frac{\left\| {\mathbb{A}_v(A,B;X)} \right\|_F^2-\left\| \mathbb{H}_v(A,B;X) \right\|_F^2}{\left\| {\mathbb{A}_\tau(A,B;X)} \right\|_F^2-\left\| \mathbb{H}_\tau(A,B;X)\right\|_F^2}\le\left(\frac{1-v}{1-\tau}\right)^2,
\end{equation}
where
\begin{align*}
\mathbb{A}_v(A,B;X)&=vAX +(1 - v)XB,\\
 \mathbb{H}_v(A,B;X)&=\left[vX^{-1}A^{-1}+ (1 - v)B^{-1}X^{-1}\right]^{-1}.
 \end{align*}
\end{theorem}

\begin{proof}
Since $A$ and $B$ are positive definite, it follows by the
spectral theorem that there exist unitary matrices $U,V \in M_n ({\rm
C})$ such that
$$
A = U\Lambda _1 U^ * ,
B = V\Lambda _2 V^ * ,
$$
where the two matrices $\Lambda _1 = {\rm diag}(\mu _1 ,\mu _2 , \cdots ,\mu _n
),$ $\Lambda _2 = {\rm diag}(\nu _1 ,\nu _2 , \cdots ,\nu _n )$ and $\mu _i , \nu _i \ge
0,$ $i = 1,2,\cdots,n.$

Let $Y = U^ * XV = [y_{ij} ]$, then
\begin{align*}
vAX +(1 - v)XB &= U(v\Lambda _1 Y + (1-v)Y\Lambda _2 )V^*\\
&=U\left[(v\mu _i + (1-v)\nu _j )y_{ij} \right]V^* ,\\
\left[vX^{-1}A^{-1}+ (1 - v)B^{-1}X^{-1}\right]^{-1}&= U\left(vY^{-1}\Lambda _1^{-1} +(1-v)\Lambda _2^{-1}Y^{-1} \right)V^*\\
&=U\left[\left(v\mu _i^{-1} + (1 - v )\nu _j^{-1} \right)^{-1}y_{ij} \right]V^* .
\end{align*}
Now using the first inequality in \eqref{th21} with $\lambda=2$ and the unitary invariance of the Hilbert-
Schmidt norm, we have
\begin{align*}
&\left\| {vAX +(1 - v)XB} \right\|_F^2-\left\| \left[vX^{-1}A^{-1}+ (1 - v)B^{-1}X^{-1}\right]^{-1} \right\|_F^2\\
&= \sum\limits_{i,j = 1}^n {(v\lambda _i + (1 - v) \nu _j )^2\vert y_{ij} \vert ^2}- \sum\limits_{i,j = 1}^n {\left(v\lambda _i^{-1} + (1 - v) \nu _j^{-1} \right)^{-2}\vert y_{ij} \vert ^2} \\
&=\sum\limits_{i,j = 1}^n \left[{(v\lambda _i +(1 - v)\nu _j )^2}-  {\left(v\lambda _i^{-1} +(1 - v) \nu _j^{-1} \right)^{-2}
}\right]\vert y_{ij} \vert ^2\\
&\ge \left(\frac{v}{\tau}\right)^2\sum\limits_{i,j = 1}^n \left[{(\tau\lambda _i +(1 - \tau)\nu _j )^2}-  {\left(\tau\lambda _i^{-1} +(1 - \tau) \nu _j^{-1} \right)^{-2}
}\right]\vert y_{ij} \vert ^2\\
&= \left(\frac{v}{\tau}\right)^2\left[\sum\limits_{i,j = 1}^n {\left(\tau\lambda _i +(1 - \tau)\nu _j \right)^2}\vert y_{ij} \vert ^2- \sum\limits_{i,j = 1}^n{\left(\tau\lambda _i^{-1} +(1 - \tau) \nu _j^{-1} \right)^{-2}}\vert y_{ij} \vert ^2\right]\\
&= \left(\frac{v}{\tau}\right)^2\left[\left\|\tau AX+(1 - \tau)XB\right\|_F^2-\left\| \left(\tau X^{-1}A^{-1}+ (1 - \tau)B^{-1}X^{-1}\right)^{-1} \right\|_F^2\right],
\end{align*}
which proves the first inequality in \eqref{th41}.

The proof of the second inequality in \eqref{th41} can be completed by an argument similar to that used in the proof of the first one.
\end{proof}

As direct conclusions of Theorem \ref{theorem41}, we have the next two corollarys.
\begin{corollary}
Let $A, B, X\in M_n(\mathbb{C})$ such that $A$ and $B$ be positive definite. If $v$ is a real number with $0\le v\le 1$, then
\begin{equation}
\label{cor41}
\begin{split}
\left\| {vAX +(1 - v)XB} \right\|_F^2&\ge\left\|A^vX B^{1 - v} \right\|_F^2\\
&\ge\left\| \left[vX^{-1}A^{-1}+ (1 - v)B^{-1}X^{-1}\right]^{-1} \right\|_F^2.
\end{split}
\end{equation}
\end{corollary}

\begin{proof}
The inequality \eqref{cor41} follows from the proof of Theorem \ref{theorem41} by the inequalities \eqref{in11}.
\end{proof}

\begin{corollary}
Let $A, B, X\in M_n(\mathbb{C})$ such that $A$ and $B$ be positive definite. If $v$ is a real number with $0\le v\le 1/2$, then

\begin{equation}
\label{cor42}
\begin{split}
4v^2&\left[\left\| \frac{AX+XB}{2}\right\|_F^2-\left\| \left(\frac{X^{-1}A^{-1}+ B^{-1}X^{-1}}{2}\right)^{-1} \right\|_F^2\right]\\
&\le\left\| {vAX +(1 - v)XB} \right\|_F^2-\left\| \left[vX^{-1}A^{-1}+ (1 - v)B^{-1}X^{-1}\right]^{-1} \right\|_F^2\\
&\le 4(1-v)^2\left[\left\| \frac{AX+XB}{2}\right\|_F^2-\left\| \left(\frac{X^{-1}A^{-1}+ B^{-1}X^{-1}}{2}\right)^{-1} \right\|_F^2\right].
\end{split}
\end{equation}
\end{corollary}

Note that \eqref{cor41} can be regarded as the arithmetic-harmonic mean inequality for
the Hilbert-Schmidt norm. The inequalities \eqref{cor42} contain the refinement and reverse of \eqref{cor41} and presented the Hilbert-Schmidt norm version of \eqref{rem21}.

\section{Determinant inequalities for arithmetic mean and harmonic mean}
In this section, the singular values of a matrix $A$ are denoted by $s_j(A), j=1,2,\cdots, n$ and we adhere to the convention that singular values are sorted in non-increasing order.
$\det(A)$ denotes the determinant of $A$.

Obviously, by \eqref{in1}, for positive definite matrices $A,B\in M_n(\mathbb{C})$ and $0\le v\le 1$,
\[
 A!_v B\le  A\nabla_v B.
\]
So we have the following proposition.
\begin{proposition}
Let $A,B\in M_n(\mathbb{C})$ be positive definite matrices and $0\le v\le 1$. Then
\begin{equation}
\label{52}
\det \left(A!_v B\right)^{\lambda}\le \det  \left(A\nabla_v B\right)^{\lambda}.
\end{equation}
\end{proposition}
\begin{proof}
Denote the positive definite matrix $T=A^{-\frac{1}{2}}BA^{-\frac{1}{2}}$, by arithmetic-harmonic mean inequality, then we have
\[
 \left(v+(1-v)s_i(T)\right)^{\lambda}\ge\left(v+(1-v)s_i(T)^{-1}\right)^{-\lambda}
\]
for all $i=1,2,\cdots,n$.
\begin{align*}
\det(v+(1-v)T)^{\lambda}
&=\prod\limits_{i=1}^n(v+(1-v)s_i(T))^{\lambda}\\
&\ge\prod\limits_{i=1}^n\left(v+(1-v)s_i(T)^{-1}\right)^{-\lambda}\\
&=\det(v+(1-v)T^{-1})^{-\lambda}.
\end{align*}
Multiplying $\left(\det A^{1/2}\right)^{\lambda}$ to the both sides, we deduce the result by the multiplicativity of the determinant.
\end{proof}

Next, we will improve the inequality \eqref{52}, the following two lemmas must be mentioned.
\begin{lemma}{\rm(Minkowski's product inequality \cite[p.560]{Horn})} Let $a=[a_i],b=[b_i], i=1,2,\cdots,n$ such that $a_i$ and $b_i$ be positive real numbers. Then
\[
\left(\prod\limits_{i=1}^na_i\right)^{\frac{1}{n}}+\left(\prod\limits_{i=1}^n b_i\right)^{\frac{1}{n}}\le \left(\prod\limits_{i=1}^n(a_i+b_i)\right)^{\frac{1}{n}}.
\]
Equality holds if and only if a=b.
\end{lemma}

\begin{lemma}\label{lemma52}
Let $a$ and $b$ be positive real numbers with $a>b$. If $\lambda\ge1$, then
\[
a^{\lambda}-b^{\lambda}\ge (a-b)^{\lambda}.
\]

\end{lemma}

\begin{theorem}\label{theorem51}
Let $A, B\in M_n(\mathbb{C})$ be positive definite. If $v, \tau$ and $\lambda$ are real numbers with $0<v\le \tau<1$ and $\lambda\ge1$, then
\begin{equation}
\label{53}
\left(\frac{v}{\tau}\right)^{\lambda}\det\left(A\nabla_\tau B-A!_\tau B\right)^{\frac{\lambda}{n}}\le\det \left(A\nabla_v B\right)^{\frac{\lambda}{n}}-\det \left(A!_v B\right)^{\frac{\lambda}{n}}.
\end{equation}
\end{theorem}

\begin{proof}
By the first inequality of \eqref{th21} and we denote the positive definite matrix $T=A^{-\frac{1}{2}}BA^{-\frac{1}{2}}$,
\[
\left(\frac{v}{\tau}\right)^{\lambda}\le \frac{\left(v+(1-v)s_i(T)\right)^{\lambda}-\left(v+(1-v)s_i(T)^{-1}\right)^{-\lambda}}{\left(\tau+(1-\tau)s_i(T)\right)^{\lambda}-\left(\tau+(1-\tau)s_i(T)^{-1}\right)^{-\lambda}}
\]
for all $i=1,2,\cdots,n$ for which $s_i(T)\neq 1$. Since the determinant of a positive definite matrix is product of its singular values. Thus
\begin{align*}
\det &\left(vI+(1-v)T\right)^{\frac{\lambda}{n}}\\
&=\left(\det\left(vI+(1-v)T\right)^{\lambda}\right)^{\frac{1}{n}}\\
&=\left(\prod\limits_{i=1}^n\left(vI+(1-v)s_i(T)\right)^{\lambda}\right)^{\frac{1}{n}}=\left(\prod\limits_{i=1}^nA_v(1,s_i(T))^{\lambda}\right)^{\frac{1}{n}}\\
&\ge\left(\prod\limits_{i=1}^n\left[\left(\frac{v}{\tau}\right)^{\lambda}\left(A_\tau(1,s_i(T))^{\lambda}-H_\tau(1,s_i(T))^{\lambda}\right)+H_v(1,s_i(T))^{\lambda}\right]\right)^{\frac{1}{n}}\\
&\ge\left(\frac{v}{\tau}\right)^{\lambda}\prod\limits_{i=1}^n\left[A_\tau(1,s_i(T))^{\lambda}-H_\tau(1,s_i(T))^{\lambda}\right]^{\frac{1}{n}}+\prod\limits_{i=1}^n\left[H_v(1,s_i(T))^{\lambda}\right]^{\frac{1}{n}}\\
&\ge\left(\frac{v}{\tau}\right)^{\lambda}\prod\limits_{i=1}^n\left[\left(A_\tau(1,s_i(T))-H_\tau(1,s_i(T))\right)^{\lambda}\right]^{\frac{1}{n}}+\prod\limits_{i=1}^n\left[H_v(1,s_i(T))^{\lambda}\right]^{\frac{1}{n}}\\
&=\left(\frac{v}{\tau}\right)^{\lambda}\det\left[\left(I\nabla_\tau T-I!_\tau T\right)^{\lambda}\right]^{\frac{1}{n}}+\left[\det\left(I!_v T\right)^{\lambda}\right]^{\frac{1}{n}}.
\end{align*}
The second inequality is obtained by Minkowski's product inequality and the last inequality is obtained by Lemma \ref{lemma52}.
Multiplying $\left(\det A^{1/2}\right)^{\lambda/n}$ to the both sides and by the multiplicativity of the determinant, we derive \eqref{53}.
\end{proof}

\begin{remark}
If $\lambda=1$ in the inequality \eqref{53}, then
\[
\frac{v}{\tau}\det\left(A\nabla_\tau B-A!_\tau B\right)^{\frac{1}{n}}\le\det \left(A\nabla_v B\right)^{\frac{1}{n}}-\det \left(A!_v B\right)^{\frac{1}{n}}.
\]
If $\lambda=1$ and $\tau=1/2$ in the inequality \eqref{53}, then
\[
2v\det\left(A\nabla B-A! B\right)^{\frac{1}{n}}\le\det \left(A\nabla_v B\right)^{\frac{1}{n}}-\det \left(A!_v B\right)^{\frac{1}{n}}.
\]
\end{remark}

\begin{corollary}
Let $A, B\in M_n(\mathbb{C})$ be positive definite. If $v, \tau$ are real numbers with $0<v\le \tau<1$, then
\begin{equation}
\label{54}
\det A!_v B+\left(\frac{v}{\tau}\right)^n\det\left(A\nabla_\tau B-A!_\tau B\right)\le \det A\nabla_v B.
\end{equation}
\end{corollary}

\begin{proof}
The proof is similar to that of Theorem \ref{theorem51}.
\end{proof}
\begin{corollary}
Let $A, B\in M_n(\mathbb{C})$ be positive definite. If $v$ is real number with $0\le v\le1/2$, then
\begin{equation}
\label{55}
\det A!_v B+(2v)^n\det\left(A\nabla B-A! B\right)\le \det A\nabla_v B.
\end{equation}
\end{corollary}

Note that the inequality \eqref{55} is the determinant version of the inequality \eqref{zuo2}. The inequalities \eqref{53} and \eqref{54} can be treated as two generalizations of \eqref{55}.


\bibliographystyle{amsplain}

\end{document}